\documentclass[12pt]{extarticle}

\usepackage[margin=1.5in]{geometry}  
\usepackage{graphicx}              
\usepackage{amsmath}               
\usepackage{amsfonts}              
\usepackage{amsthm}                
\usepackage{xcolor}				   
\usepackage{hyperref}			   

\newtheorem{thm}{Theorem}
\newtheorem{lem}[thm]{Lemma}

\begin{document}

\title{\bf A Conditional Explicit Result \\ for the Prime Number Theorem \\ in Short Intervals}

\author{\textsc{Michaela Cully-Hugill} \\ School of Science, UNSW Canberra \\ Australia ACT 2612 \\   \ \\\textsc{Adrian W. Dudek} \\ Wacal Road, Mothar Mountain \\ Australia QLD 4570}

\maketitle

\begin{abstract}
This paper gives an explicit bound for the prime number theorem in short intervals under the assumption of the Riemann hypothesis.
\end{abstract}

\section{Introduction}
The von Mangoldt function is defined as
\begin{displaymath}
   \Lambda(n) = \left\{
     \begin{array}{ll}
       \log p  & : \hspace{0.1in} n=p^m, \text{ $p$ is prime, $m \in \mathbb{N}$}\\
       0   & : \hspace{0.1in} \text{otherwise,}
     \end{array}
   \right.
\end{displaymath} 
and we will consider the sum $\psi(x) = \sum_{n \leq x} \Lambda(n)$. The prime number theorem (PNT) is the statement $\psi(x) \sim x$ as $x \rightarrow \infty$. For the PNT in short intervals, it is known that 
\begin{equation}\label{PNT-intervals}
\psi(x+h) - \psi(h) \sim h
\end{equation}
provided that $h$ grows suitably with respect to $x$. Heath-Brown \cite{heathbrown} has shown that one can take $h = x^{\frac{7}{12} - \epsilon}$ provided that $\epsilon \rightarrow 0$ as $x \rightarrow \infty$. Assuming the Riemann hypothesis (RH), Selberg \cite{selberg} showed that (\ref{PNT-intervals}) is true for any $h=h(x)$ such that $h/(x^{1/2} \log x) \rightarrow \infty$ as $x \rightarrow \infty$. On the other hand, Maier \cite{maier} has shown that the statement is false for $h = (\log x)^{\lambda} $ for any $\lambda > 1$.

In this paper we prove the following explicit version of Selberg's result.
\begin{thm} \label{main}
Assuming RH, for any $h$ satisfying $\sqrt{x}\log x\leq h \leq x^\frac{3}{4}$ and all $x \geq e^{10}$ we have
\begin{equation}
|\psi(x+h) - \psi(x)-h|< \frac{1}{\pi} \sqrt{x} \log x \log\bigg( \frac{h}{\sqrt{x}\log x} \bigg) + 2\sqrt{x} \log x.
\end{equation} 
\end{thm}

Selberg's result follows from Theorem 1 for any $h = f(x)\sqrt{x}\log x$ with unbounded $f(x)=o(x)$, in that we would have
$$|\psi(x+h) - \psi(x)-h| \ll \sqrt{x} \log x \log\left( f(x) \right) = o(h).$$
For $h = c \sqrt{x} \log x$, Theorem 1 implies Cram{\'e}r's \cite{cramerorder} result on primes in the interval $(x, x + h)$ for all sufficiently large $x$ and $c$. In an earlier paper \cite{dudekcramer}, the author showed that $c = 1 + \epsilon$ is suitable for any $\epsilon > 0$ and for all sufficiently large $x$. Carneiro, Milinovich and Soundararajan \cite{carneiro} have since shown that we can take $c=22/55$ for all $x\geq 4$. The same methods used in \cite{dudekcramer} are applied to reach Theorem 1. As such, it could be possible to sharpen Theorem \ref{main} using the techniques in \cite{carneiro}.

The closest result to Theorem \ref{main} is the following from Schoenfeld \cite{schoenfeldjust}.

\begin{thm}
Assuming RH, for $x \geq 73.2$ we have
\begin{equation} \label{schoenfeld}
|\psi(x) - x|< \frac{1}{8 \pi} \sqrt{x} \log^2 x.
\end{equation} 
\end{thm}

Schoenfeld's result confirms Selberg's theorem for the slightly stronger condition of $h/(\sqrt{x} \log^2 x) \rightarrow \infty$. One also has from the above
$$|\psi(x+h) - \psi(x)-h | < \frac{1}{4 \pi} \sqrt{x+h} \log^2 (x+h).$$
When $x$ is sufficiently large, Theorem \ref{main} improves the leading constant in this bound for any choice of $h\leq x^{0.735}$.

\section{Proof of Theorem \ref{main}}

\subsection{A smooth explicit formula}

The Riemann--von Mangoldt explicit formula relates $\psi(x)$ to the zeros of the Riemann zeta-function $\zeta(s)$ (e.g. see Ingham \cite{inghambook}). Tor all non-integer $x>0$, 
\begin{equation} \label{explicitoriginal}
\psi(x) = x - \sum_{\rho} \frac{x^\rho}{\rho}-\log 2\pi - \frac{1}{2} \log(1-x^{-2}),
\end{equation}
where the sum is over all non-trivial zeroes $\rho = \beta+i\gamma$ of $\zeta(s)$. We define the weighted sum
\begin{equation}\label{weighted}
\psi_1 (x) = \sum_{n \leq x} (x-n) \Lambda(n) = \int_2^x \psi(t) dt
\end{equation}
and use the following explicit formula, proved in \cite{dudekcramer} (see also Thm. 28 of \cite{inghambook}).
 
\begin{lem} \label{first}
For non-integer $x>0$ we have

\begin{equation} \label{explicit}
\psi_1(x) = \frac{x^2}{2} - \sum_{\rho} \frac{x^{\rho+1}}{\rho (\rho +1)} - x \log(2\pi) + \epsilon(x)
\end{equation}
where
$$1.545 < \epsilon(x) < 2.069.$$
\end{lem}

\noindent The bound on $\epsilon(x)$ has been reduced from \cite{dudekcramer}, as we can write
\begin{align*}
\epsilon(x) &= 2\log 2\pi - 2  +  \sum_\rho \frac{2^{\rho+1}}{\rho(\rho+1)} - \frac{1}{2} \int_2^x \log(1-t^{-2}) dt \\
&< 2\log 2\pi - 2 + 2^\frac{3}{2}(\gamma + 2 - \log 4\pi) + \log \frac{3\sqrt{3}}{4} < 2.069
\end{align*}
and $$\epsilon(x) > 2\log 2\pi - 2 - 2^\frac{3}{2}(\gamma + 2 - \log 4\pi) > 1.545.$$

Using a linear combination of equation (\ref{weighted}), we can examine the distribution of prime powers in the interval $(x, x+h)$. For $2\leq \Delta < \sqrt{x}\log x \leq h \leq x$, let

\begin{displaymath}
   w(n) = \left\{
     \begin{array}{ll}
     (n-x+\Delta)/\Delta & : \hspace{0.1in}  x-\Delta \leq n \leq x\\
       1 & : \hspace{0.1in}  x \leq n \leq x+h\\
      (x+h+\Delta-n)/\Delta & : \hspace{0.1in}  x+h \leq n \leq x+h+\Delta\\
       0   & : \hspace{0.1in} \text{otherwise.}
     \end{array}
   \right.
\end{displaymath} 
This leads to the identity

\begin{eqnarray*} \label{weighted-identity}
\sum_{n} \Lambda(n) w(n) & = & \frac{1}{\Delta} (\psi_1(x+h+\Delta) - \psi_1(x+h) - \psi_1 (x) + \psi_1(x-\Delta)),
\end{eqnarray*}
which can be verified by expanding both sides. Notice that over $x \leq n \leq x+h$, the sum on the LHS is equal to $\psi(x+h)-\psi(x)$. We thus aim to estimate this expression by bounding the RHS of (\ref{weighted-identity}). Using Lemma \ref{first} in the above equation gives the following.

\begin{lem} \label{dog}
Let $2 \leq \Delta < h \leq x$ with $x \notin \mathbb{Z}$. Then

$$\sum_{n} \Lambda(n) w(n) = h + \Delta - \frac{1}{\Delta} \sum_{\rho} S(\rho) + \epsilon(\Delta)$$
where 

$$S(\rho) =\frac{ (x+h+\Delta)^{\rho+1} - (x+h)^{\rho+1} - x^{\rho+1} +(x-\Delta)^{\rho+1}}{\rho(\rho+1)}$$
and 

$$|\epsilon(\Delta)| < \frac{21}{20\Delta}.$$

\end{lem}

It remains to estimate the sum over zeros. We will split it into three sums,
\begin{equation}\label{three}
\sum_{\rho} S(\rho) = \bigg( \sum_{| \gamma | \leq \alpha x/h} +\sum_{\alpha x/h <| \gamma | < \beta x/\Delta} + \sum_{  | \gamma | \geq \beta x/\Delta }  \bigg)  S(\rho)
\end{equation}
where $\alpha > 0$ and $\beta>0$ are parameters we can later optimise over.

\begin{lem} \label{highbound}
Let $2 \leq \Delta < h \leq x$ and assume RH. We have
$$\left| \sum_{| \gamma | \geq \beta x/\Delta} S(\rho) \, \right| < \frac{4 \Delta (x+h + \Delta)^{3/2} }{\pi \beta x}\log(\beta x/\Delta) $$
provided that $\beta x / \Delta \geq \gamma_1 = 14.13\ldots$, the ordinate of the first zero of $\zeta(s)$.
\end{lem}

\begin{proof}
On RH, one has
$$|S(\rho)| \leq \frac{4 (x+h+\Delta)^{3/2}}{\gamma^2}.$$
The result follows from Lemma 1(ii) of Skewes \cite{skewes}, that for all $T \geq \gamma_1$,
$$\sum_{\gamma \geq T} \frac{1}{\gamma^2} < \frac{1}{2\pi} \frac{\log T}{T}.$$

\end{proof}

The following lemmas require estimates on the zero-counting function $N(T)$, which counts the number of zeros of $\zeta(s)$ in the critical strip $0<\beta<1$ with $0<\gamma\leq T$. Backlund \cite{Backlund_18} showed that $N(T) = P(T) + Q(T)$, where
\begin{equation*}
    P(T) := \frac{T}{2\pi}\log{\frac{T}{2\pi}} - \frac{T}{2\pi} + \frac{7}{8}
\end{equation*}
and $Q(T) = O(\log T)$. Hasanalizade, Shen, and Wong \cite[Cor.~1.2]{H_S_W_22} have given the most recent explicit version of this, of
\begin{align}\label{R(T)}
|Q(T)| \leq R(T) = a_1\log{T} + a_2\log\log{T} + a_3
\end{align} 
with $a_1= 0.1038$, $a_2=0.2573$, and $a_3 = 9.3675$, for all $T \geq e$.

\begin{lem} \label{lowbound}
Let $2 \leq \Delta < h \leq x$ and assume RH. We have
$$\left| \sum_{| \gamma | \leq \alpha x/h} S(\rho) \, \right| < \frac{\alpha x (h+\Delta) \Delta}{ \pi h \sqrt{x-\Delta}}  \log(\alpha x/h).$$
\end{lem}

\begin{proof}
We can write
$$S(\rho) = \int_{x+h}^{x+h+\Delta} \int_{u-h-\Delta}^u t^{\rho-1} dt du,$$
so, under RH, one has
$$|S(\rho)| < \frac{ (h+\Delta) \Delta}{  \sqrt{x-\Delta}}. $$
With (\ref{R(T)}), we can use 
$$N(T) < \frac{T \log T}{2 \pi},$$
from which the result immediately follows.
\end{proof}

For the middle sum of (\ref{three}), we will use the following lemma. It follows directly from Lemma 3 of \cite{B_P_T_21}, in whose notation we use $\phi(\gamma) = \gamma^{-1}$, and takes constants $A_0$ and $A_1$ from Trudgian \cite[Thm.~2.2]{Trudgian_11} and $A_2$ from \cite[Lem.~2]{B_P_T_21}.

\begin{lem}\label{reciprocalzerobound}
For $2\pi \leq T_1 \leq T_2$ we have
\begin{align}
\sum_{T_1 < \gamma < T_2} \frac{1}{\gamma} &= \frac{1}{4 \pi}\log \frac{T_2}{T_1} \log \frac{T_2 T_1}{4\pi^2}  + \frac{Q(T_2)}{T_2} - \frac{Q(T_1)}{T_1} + E(T_1),
\end{align}
where $|Q(T)|\leq R(T)$, defined in (\ref{R(T)}), and
\begin{align*}
|E(T)| \leq \frac{2A_1\log T + 2A_0 + A_1 + A_2}{T^2}
\end{align*}
with $A_0 = 2.067$, $A_1 = 0.059$, $A_2 = 1/150$.
\end{lem}

\begin{lem} \label{middlebound}
Let $2 \leq \Delta < h \leq x$ and assume RH. For $\alpha x/h\geq 15$ we have
$$\left| \sum_{\alpha x/h <| \gamma | < \beta x/\Delta} S(\rho) \, \right|  < \Delta (x+h+\Delta)^{1/2}  \bigg( \frac{1}{\pi}\log \bigg( \frac{\beta h}{\alpha \Delta} \bigg) \log\bigg( \frac{\alpha \beta x^2}{4\pi^2 h \Delta} \bigg) + 5.4\bigg).$$
\end{lem}

\begin{proof}
We can write
$$S(\rho) = \frac{1}{\rho} \bigg( \int_{x+h}^{x+h+\Delta} t^{\rho} dt - \int_{x-\Delta}^x t^{\rho} dt \bigg),$$
and so bounding trivially gives
$$| S(\rho)| \leq \frac{2 (x+h+\Delta)^{1/2} \Delta}{| \gamma |}.$$
It follows that
$$\left| \sum_{\alpha x/h <| \gamma | < \beta x/\Delta} S(\rho) \, \right| \leq 4 (x+h+\Delta)^{1/2} \Delta \sum_{\alpha x/h < \gamma < \beta x /\Delta} \frac{1}{\gamma},$$
on which we apply Lemma \ref{reciprocalzerobound}, and bound the smaller order terms with the assumption of $T_1\geq 15$ to obtain the result. Note that the bound on $T_1$ is to reduce the constant $5.4$, but not restrict $\alpha$ too much.
\end{proof}

\subsection{Bounding the PNT in intervals}

From Lemma \ref{dog} we can write
\begin{align*}
\bigg| \psi(x+h) - \psi(x) - h  \bigg| < \ & \frac{1}{\Delta} \Bigg| \sum_{\rho} S(\rho) \Bigg|  + \Delta  + \frac{21}{20 \Delta} + \sum_{x-\Delta < n \leq x} w(n) \Lambda(n)  \\
& + \sum_{x+h < n \leq x+h+\Delta} w(n) \Lambda(n)
\end{align*}
As the smooth weight has $|w(n)| \leq 1$, the above bound is no greater than
\begin{align}\label{together}
\frac{1}{\Delta} \Bigg| \sum_{\rho} S(\rho) \Bigg|  + \Delta  + \frac{21}{20 \Delta} + 2\sum_{\substack{x+h < p^k \leq x+h+\Delta \\ k\geq 1}} \log p.
\end{align}
The largest term in this bound comes from the sum over $\rho$, in particular, the section estimated in Lemma \ref{middlebound}. Larger $\Delta$ results in a smaller main-term constant, so we will set $\Delta = C\sqrt{x}\log x$ and later choose an optimal value of $C\in(0,1)$. The reason for not taking larger $\Delta$ is two-fold: to keep $\Delta < h$ and ensure the smaller terms in (\ref{together}) are $O(\sqrt{x}\log x)$.

To bound the sum over prime powers we can use Montgomery and Vaughan's version of the Brun--Titchmarsh theorem for primes in intervals \cite[Eq.~1.12]{montgomeryvaughan}. Defining $\theta(x) = \sum_{p\leq x}\log p$, equation (1.12) of \cite{montgomeryvaughan} implies
\begin{align*}
\theta(x+h) - \theta(x) = \sum_{x<p\leq x+h} \log p \leq  \frac{2h\log (x+h)}{\log h}.
\end{align*}
The contribution from higher prime powers is relatively small, and can be bounded with explicit estimates on the difference between the Chebyshev functions $\psi(x)$ and $\theta(x)$. Costa Pereira \cite[Thm.~2,4,5]{Costa} gives lower bounds for different ranges of $x$. These can be combined into
\begin{equation}\label{Costa_diff}
    \psi(x)-\theta(x) > 0.999x^\frac{1}{2} + \frac{2}{3} x^\frac{1}{3}
\end{equation}
for all $x\geq 2187$. Broadbent \textit{et al.} \cite[Cor.~5.1]{Broadbent} give
\begin{equation}\label{Broadbent_diff}
    \psi(x) - \theta(x) < \alpha_1 x^\frac{1}{2} + \alpha_2 x^\frac{1}{3}
\end{equation}
with $\alpha_1= 1+ 1.93378 \cdot 10^{-8}$ and $\alpha_2 = 2.69$ for all $x\geq e^{10}$. Thus, we have
\begin{align*}
\psi(x+h+\Delta) - \psi(x+h) &\leq  \theta(x+h+\Delta) - \theta(x+h) + E_1(x) \\
&\leq  \frac{2\Delta\log (x+h+\Delta)}{\log \Delta} + E_1(x)
\end{align*}
where $E_1(x) = \alpha_1 (x+h+\Delta)^\frac{1}{2} + \alpha_2 (x+h+\Delta)^\frac{1}{3} - 0.999(x+h)^\frac{1}{2} - \frac{2}{3} (x+h)^\frac{1}{3}$, and is bounded by $E_1(x) \leq \beta_1 x^{\frac{1}{2}} + \beta_2 x^{\frac{1}{3}}$ with $$\beta_1 = \sqrt{3}\alpha_1 - 0.999 \quad \text{and} \quad \beta_2 = 3^{\frac{1}{3}}\alpha_2 - \frac{2}{3}.$$

Here and hereafter, let $x_0=e^{10}$. For $x\geq x_0$ we can bound the smaller order terms in (\ref{together}),
\begin{align*}
\Delta  + \frac{21}{20 \Delta} + 2\sum_{\substack{x+h < p^k \leq x+h+\Delta \\ k\geq 1}} \log p < K_1 \sqrt{x}\log x
\end{align*}
where, for $h\leq x^t$ with $t<1$,
\begin{align*}
K_1 &= C + \frac{4C \log (x_0+2x_0^t)}{\log  (C \sqrt{x_0}\log x_0)} + \frac{2\beta_1}{\log x_0} + \frac{2\beta_2}{x_0^{\frac{1}{6}}\log x_0} + \frac{21}{20C x_0\log^2 x_0}.
\end{align*}
This, along with Lemmas \ref{highbound} and \ref{lowbound}, allow us to bound
\begin{align}\label{init}
\bigg| \psi(x+h) - \psi(x) - h  \bigg| < \frac{1}{\Delta} \Bigg| \sum_{\alpha x/h < |\gamma| < \beta x / \Delta} S(\rho) \Bigg| + E(x,h,\Delta)
\end{align}
where
\begin{align*}
E(x,h,\Delta) &= K_1 \sqrt{x} + \frac{\alpha x (h + \Delta)}{\pi h \sqrt{x-\Delta} } \log\left(\frac{\alpha x}{h}\right) + \frac{4 (x+h+\Delta)^{3/2}}{\pi \beta x} \log\left(\frac{\beta x}{\Delta}\right).
\end{align*}
For $\sqrt{x}\log x\leq h\leq x^t$ we have
\begin{align*}
E(x,h,\Delta) \leq \ & K_1 \sqrt{x} + \frac{2\alpha x}{\pi \sqrt{x-C\sqrt{x}\log x}} \log\left( \frac{\alpha \sqrt{x}}{\log x}\right) \\
& + \frac{4 (x+x^t+C\sqrt{x}\log x)^{3/2}}{\pi \beta x} \log\left(\frac{\beta \sqrt{x}}{C\log x}\right) \leq K_2 \sqrt{x}\log x,
\end{align*}
where, for $x\geq x_0\geq e^{\beta/C}$ and $0<\alpha\leq 5$, we can take
\begin{align*}
K_2 = \ & \frac{K_1}{\log x_0} + \frac{\alpha }{\pi} + \frac{2 (x_0+x_0^t + C\sqrt{x_0}\log x_0)^{3/2}}{\pi \beta x_0^{3/2}}.
\end{align*}

The first term in (\ref{init}) can be estimated with Lemma \ref{middlebound}, so that
\begin{align*}
\frac{1}{\Delta} \Bigg| \sum_{\alpha x/h <| \gamma | < \beta x/\Delta} S(\rho) \Bigg| &< (x+h+\Delta)^{1/2}  \bigg( \frac{1}{\pi}\log \bigg( \frac{\beta h}{\alpha \Delta} \bigg) \log\bigg( \frac{\alpha \beta x^2}{4\pi^2 h \Delta} \bigg) + 5.4\bigg) \\
&< \frac{\sqrt{x}}{\pi}\log x\log\left( \frac{h}{\sqrt{x}\log x}\right) + K_3 \sqrt{x} \log x,
\end{align*}
in which, assuming $100 e^{-10} \leq \frac{\alpha\beta}{4\pi^2 C}\leq 100$, we can take 
\begin{align*}
K_3 &= \frac{1}{\pi} \log\left(\frac{\beta}{\alpha C} \right)\log\left( \frac{\alpha \beta x_0}{4\pi^2 C\log^2 x_0} \right)\frac{1}{\log x_0} \\
&\quad + \frac{x_0^{t/2-1/2}}{\pi\log x_0}\log \bigg( \frac{\beta x_0^{t-1/2}}{\alpha C\log x_0} \bigg) \log\bigg( \frac{\alpha \beta x_0}{4\pi^2 C\log^2 x_0} \bigg) \\
&\quad + \frac{\sqrt{C}}{\pi x_0^{1/4} \sqrt{\log x_0}} \log \bigg( \frac{\beta x_0^{t-1/2}}{\alpha C\log x_0} \bigg) \log\bigg( \frac{\alpha \beta x_0}{4\pi^2 C\log^2 x_0} \bigg) \\
&\quad + \frac{5.4}{\log x_0} \left( 1 + x_0^{t-1} + \frac{C\log x_0}{\sqrt{x_0}} \right)^{1/2}.
\end{align*}
Note that the assumption for $\alpha$ and $\beta$ is to ensure certain terms are bounded for all $x\geq x_0$. Combining estimates, we have
\begin{align}\label{final}
\bigg| \psi(x+h) - \psi(x) - h  \bigg| <& \frac{\sqrt{x}}{\pi}\log x\log\left( \frac{h}{\sqrt{x}\log x}\right) + K_4 \sqrt{x}\log x,
\end{align}
where $K_4 = K_3 + K_2$. It remains to optimise over the parameters. Before deciding these values, recall that we have made the assumptions $\beta \leq 10C$,
\begin{align*}
\frac{15 h}{x} \leq \alpha \leq 5,\quad \beta\geq \gamma_1 \frac{C\log x}{\sqrt{x}},\quad C \alpha < \beta\leq \alpha, \quad \text{and}\quad  \frac{100}{e^{10}}\leq \frac{\alpha\beta}{4\pi^2 C}\leq 100.
\end{align*}
The restriction on $\alpha$ will be satisfied for all $\sqrt{x}\log x\leq h\leq x^\frac{3}{4}$ if we take $\alpha\geq 15 x_0^{-\frac{1}{4}}$. Optimising over $C$, $\alpha$, and $\beta$ to minimise $K_4$, we find that choosing $C= 0.25$ and $\alpha = \beta = 1.35$ allows us to take $K_4 = 2$ for all $x\geq x_0$.

\bibliographystyle{plain}

\bibliography{biblio}

\end{document}